\documentclass[12 pt]{article}        	
\usepackage{amsfonts, amssymb}					 
\usepackage[colorlinks=true,citecolor=teal,linkcolor=blue]{hyperref}
\usepackage{tikz}
\usepackage{tikz-cd}
\usepackage{quiver} 
\usepackage{xcolor}
 \usepackage{graphicx} 
 \usepackage{amsmath}
 \usepackage{mathrsfs}
\usepackage{mathabx}
\usepackage{comment}
\usepackage{setspace} 
\usepackage{enumerate}
\usepackage{amsthm}%

\usepackage{url}

\newcommand{\eqn}[0]{\begin{array}{rcl}}
\newcommand{\eqnend}[0]{\end{array} }  	
\newcommand\norm[1]{\left\lVert#1\right\rVert}
\newcommand\restr[2]{{
  \left.\kern-\nulldelimiterspace 
  #1
  \vphantom{\big|} 
  \right|_{#2} 
  }}
  
\def\P{\mathbb{P}}
\def\E{\mathbb{E}}
\def\Z{\mathbb{Z}}
\def\N{\mathbb{N}}
\def\R{\mathbb{R}}
\numberwithin{equation}{section}
\begin{document}	
\newtheorem{theorem}{Theorem}[section]
\newtheorem{prop}[theorem]{Proposition}

\theoremstyle{definition}
\newtheorem{definition}[theorem]{Definition}

\newtheorem{assumption}[theorem]{Assumption}

\theoremstyle{definition}
\newtheorem{remark}[theorem]{Remark}

\theoremstyle{definition}
\newtheorem{problem}[theorem]{Problem}

\theoremstyle{definition}
\newtheorem{example}[theorem]{Example}

\newtheorem{corollary}[theorem]{Corollary}

\newtheorem{lemma}[theorem]{Lemma}

\title{On the Rigidity of Projected Perturbed Lattices}
\author{Youssef Djellouli\footnote{Syracuse University, Department of Mathematics. Email:
\href{mailto:mydjello@syr.edu}{mydjello@syr.edu}.} \and Pierre Yves Gaudreau Lamarre\footnote{Syracuse University, Department of Mathematics. Email: \href{mailto:pgaudrea@syr.edu}{pgaudrea@syr.edu}.}}
\date{}
\maketitle

\begin{abstract}
We study the occurrence of number rigidity and deletion singularity in a class of point processes that we call {\it projected perturbed lattices}. These are generalizations of processes of the form $\Pi=\{\|z\|^\alpha+g_z\}_{z\in\Z^d}$ where $(g_z)_{z\in\Z^d}$ are jointly Gaussian, $\alpha>0$, $d\in\N$, and $\|\cdot\|$ is a norm. We develop a new technique to prove sufficient conditions for the deletion singularity of $\Pi$, which improves significantly on the conditions one can obtain using the standard rigidity toolkit (e.g., the variance of linear statistics). In particular, we obtain the first lower bounds on $\alpha$ for the deletion singularity of $\Pi$ that are independent of the dimension $d$ and the correlation of the $g_z$'s.
\end{abstract}

\section{Introduction}

\subsection{General Set-Up}
Let $G$ be a countably infinite set (typically some graph or lattice). We are interested in point processes of the form $\Pi=\{V(z)+g_{z}\}_{z\in G}$, where $V:G\to\R$ is a deterministic function and $\{g_z\}_{z\in G}$ are random variables such that $\Pi$ is locally finite almost surely. For $d\in\N$, a recurrent example will be the projected perturbed lattice model
\begin{align}
\label{eqn: Zd model}
\{\norm{z}^\alpha +g_z\}_{z\in\Z^d},
\end{align}
where $\norm{\cdot}$ is some norm on $\R^d$, $\alpha>0$ is fixed, and $(g_z)_{z\in\Z^d}$ is a gaussian process. This is motivated by its use as a toy model of the eigenvalue process of the random Schrödinger operator $\hat{H}=-\Delta+\norm{z}^\alpha+g_z$ using the standard Laplacian on $\Z^d$ (our hope is that the methods in this paper could eventually be applied to this model).

In this paper, we are interested in studying the occurrence of the following closely related notions of ``rigidity" in these point processes:

\begin{definition}
For any $S\subset G$, denote $\Pi_S=\{V(z)+g_{z}\}_{z\in G\setminus S}$, i.e., $\Pi$ with the points coming from $S$ removed.
\begin{itemize}
\item $k$-deletion tolerance/singularity \cite{HolroydSoo}: Let $k\in\N$. We say that $\Pi$ is $k$-deletion tolerant (resp. $k$-deletion singular), if for any $S\subset G$ such that $|S|=k$, the laws of $\Pi$ and $\Pi_S$ are mutually absolutely continuous (resp. singular).
If $\Pi$ is $k$-deletion tolerant (resp. singular) for every $k\geq1$,
then we simply say that it is deletion tolerant (resp. singular).
We can think of deletion singularity as the ability to almost-surely detect if a set of points was deleted from the process, and deletion tolerance as the inability to do so.
\item Number Rigidity \cite{GhoshPeres}: We say that $\Pi$ is number rigid if for every bounded Borel set $B\subset\R$ there exists a measurable function $F$  such that $F(\Pi\cap B^c)=|\Pi\cap B|$ almost surely. In words, this consists in the ability to almost-surely determine how many points are in $B$ only by observing the configuration of points outside $B$.
\end{itemize}
\end{definition}

See, for instance, \cite{HolroydSoo,GhoshKrishnapur,GhoshPeres} for additional notions of ``rigidity" such as insertion tolerance/singularity, and rigidity for particular functions. The relationship between deletion singularity and number rigidity has been studied extensively. For instance, the following is a special case of \cite[Proposition 1.2]{PeresSly}:
\begin{prop}[\cite{PeresSly}]
\label{prop: equivalence of singularity and rigidity}
If the random variables $(g_z)_{z\in G}$ are independent and each have a density that is positive on all of $\R$, then
number rigidity is equivalent to deletion singularity.
\end{prop}

The notions of rigidity considered in this paper were first introduced in
\cite{HolroydSoo,GhoshPeres}. Since these original developments, these notions have been studied very intensively. For instance, various combinations of rigidity and deletion singularity have been proved for
determinantal and Pfaffian point processes (e.g., \cite{BufetovAiry,Buf1,Buf4,Buf2,Buf3,GhoshKrishnapur,Heisenberg,Qiu}),
point processes that arise as scaling limits of random matrix or statistical physics ensembles (e.g., \cite{ChhaibiNajnudel,Sine-BetaDLR,RigiditySAO,RigidityMSAO,Ghosh,RieszGas}),
perturbed lattice type models (e.g., \cite{PeresSly, GhoshKrishnapur}),
random Schr\"odinger spectra (e.g., \cite{RigidityMSAO,LamarreGhosalLiaoSpectralRigidity,LamarreGhosalLiaoSpacialConditioning}),
and more.

In this paper, we develop a new set of tools to study these notions for $\Pi$.

\subsection{Previous Results and Motivation} 
Two techniques have been developed and are commonly used to prove or disprove rigidity for various point processes. The first of these relies on using a pushforward to view the point process as an ordered random vector in $\R^\N$ in order to show deletion tolerance. The second technique studies the decay of the variance of certain linear test statistics to prove rigidity.
More specifically:

\begin{prop}\label{prop: Kakutani} \cite{GhoshPeres, GhoshKrishnapur}
Let us enumerate the elements of $G$ as $G=\{z_1,z_2,z_3,\ldots\}$.
Let $P=(V(z_i)+g_{z_i})_{i\geq 1}$ and $P_S=(V(z_i)+g_{z_i})_{i\geq1:z_i\not\in S}$ be random vectors in $\R^\N$ whose components are given by the points in $\Pi$ and $\Pi_S$, consistently with
the order of the enumeration of $G$'s elements. If $P$ and $P_S$ are mutually absolutely continuous, then $\Pi$ and $\Pi_S$ are mutually absolutely continuous.
\end{prop}

\begin{prop}\label{prop: variance linear statistics} \cite{GhoshPeres} If there exists a sequence of Borel measurable functions $f_n:\R\to\R$ that converge uniformly to $1$ on every compact set and such that
\[\lim_{n\to\infty}\mathrm{Var}\left(\sum\limits_{a\in\Pi}f_n(a)\right)=0,\] then $\Pi$ is number rigid
and deletion singular.
\end{prop}

Proposition \ref{prop: Kakutani} is especially powerful when the noise is independent, because it can be used in concert with Kakutani's dichotomy to give a characterization of absolute continuity that relies largely on computational methods. 
Proposition \ref{prop: variance linear statistics} is often easy to apply because
\begin{align}
\label{eqn: Linear stat variance to covariance}
\mathrm{Var}\left(\sum_{a\in\Pi}f(a)\right)=\sum_{a,b\in\Pi}\mathrm{Cov}\big(f(a),f(b)\big),
\end{align}
which further simplifies to $\sum_{a\in\Pi}\mathrm{Var}\big(f(a)\big)$ when the points in $\Pi$ are independent. Like Proposition \ref{prop: Kakutani}, this often reduces the issue of rigidity to a largely computational question that relies on pairwise correlations within $\Pi$.

\begin{remark}
\label{rmk: scaling linear statistic}
In actual applications of Proposition \ref{prop: variance linear statistics}
(e.g., \cite{LamarreGhosalLiaoSpacialConditioning,LiREU,GhoshPeres,HolroydSoo}),
it is customary to define $f_n(x)=f(x/n)$ for a smooth test function $f$ satisfying $f(0)=1$.
In many cases (e.g., Gaussian noise), this specific choice makes the computation of $\mathrm{Var}\big(f(a)\big)$ amenable to analysis by elementary methods, such as change of variables.
\end{remark}

\begin{remark}
\label{rmk: covariance trick}
If $\mathrm{Var}\left(\sum_{a\in\Pi}f_n(a)\right)$
is only bounded above as $n\to\infty$, then one can still prove deletion singularity/number rigidity if for every fixed $m\in\mathbb N$, one has
\[\lim_{n\to\infty}\mathrm{Cov}\left(\sum_{a\in\Pi}f_n(a),\sum_{a\in\Pi}f_m(a)\right)=0.\]
This more general sufficient condition was used in \cite[Lemma 7.2]{HolroydSoo} to prove the deletion singularity/number rigidity of two-dimensional perturbed lattices, and in \cite{RigiditySAO,RigidityMSAO} to prove the deletion singularity/number rigidity of stochastic Airy operators.
\end{remark}

Examples of prior results that use Propositions \ref{prop: Kakutani} and \ref{prop: variance linear statistics} to investigate the rigidity of simple special cases of $\Pi$ are as follows:

\begin{theorem}
\label{thm: Past 1}
Let $(g_k)_{k\in\Z}$ be i.i.d. standard Gaussians, and suppose that $\Pi$ is either
\begin{enumerate}
\item $\{|k|^\alpha+g_k\}_{k\in\Z}$ for some $\alpha>0$, or
\item $\{k+k^\beta g_k\}_{k\in\Z}$ for some $\beta>0$.
\end{enumerate}
\cite[Theorem 5.1]{LiREU}: In case 1., $\Pi$ is deletion tolerant if $\alpha<1/2$, and deletion singular if $\alpha>1/2$. \cite[Theorem 2.5]{GhoshKrishnapur}: In case 2, $\Pi$ is deletion tolerant if $\beta>1/2$, and $\Pi$ is deletion singular if $\beta<1/2$.

\end{theorem}
\begin{theorem}
\label{thm: Past 2}
\cite[Theorem 1.4]{PeresSly}: Let $\Pi=\{k+g_k\}_{k\in\Z}$, where $(g_k)_{k\in\Z}$ are i.i.d. symmetric $\alpha$-stable random variables. If $\alpha<1$ then $\Pi$ is deletion tolerant, and if $\alpha\geq1$ then $\Pi$ is deletion singular.
\end{theorem}

\begin{remark}
Thanks to Proposition \ref{prop: equivalence of singularity and rigidity},
Theorems \ref{thm: Past 1} and \ref{thm: Past 2} also lead to a characterization of number rigidity
for the examples of $\Pi$ considered therein.
\end{remark}

In this paper, our main motivation is to extend these results to more general instances of the model introduced in \eqref{eqn: Zd model}. In particular, we wish to understand how the choice of the dimension of the lattice $d$, the norm $\norm{\cdot}$,
the exponent $\alpha$, and the covariance function of the Gaussian process $(g_z)_{z\in\Z^d}$ each influence the rigidity of that model. As a first step in this direction, an application of
Propositions \ref{prop: Kakutani} and \ref{prop: variance linear statistics} yields:

\begin{theorem}
\label{thm: Zd Classical}
Let $\Pi$ be as in \eqref{eqn: Zd model}, assuming for simplicity that $\norm{\cdot}$ is the $\ell^1$ or $\ell^\infty$ norm, and that $(g_z)_{z\in\Z^d}$ are i.i.d. Gaussians with mean zero and variance $\sigma^2>0$. $\Pi$ is deletion tolerant if $\alpha<1/2$, and deletion singular if $\alpha> d/2$.
\end{theorem}

See Section \ref{Proof of Zd classical} for a proof.

\begin{remark}
One could also show that the process $\Pi$ in Theorem \ref{thm: Zd Classical} is deletion singular when $\alpha=d/2$ using Remark \ref{rmk: covariance trick} instead of Proposition \ref{prop: variance linear statistics},
but we do not make this precise in this paper.
\end{remark}

\begin{remark}
A similar sufficient condition for rigidity was proved when $\Pi$ is the eigenvalue point process of the Schr\"odinger operator
$-\Delta+\norm{z}^\alpha+g_z$ on $\Z^d$---see \cite[Theorem 3.16]{LamarreGhosalLiaoSpacialConditioning}; part of our motivation was to better understand
rigidity in these eigenvalue processes using \eqref{eqn: Zd model} as a toy model.
\end{remark}

\begin{remark}
\label{rmk: ellp norm 1}
In Theorem \ref{thm: Main}, if we instead assume that $\norm{\cdot}$ is the $\ell^p$ norm for some $p\in(1,\infty)$, then one can also prove that $\Pi$ is deletion singular if $\alpha\geq d/2$ using Proposition \ref{prop: variance linear statistics}. However, the sufficient condition for deletion tolerance one gets from Proposition \ref{prop: Kakutani} is much more complicated than $\alpha<1/2$. See Remark \ref{rmk: ellp norm 2} for more details.
\end{remark}

In sharp contrast to Theorems \ref{thm: Past 1} and \ref{thm: Past 2},
Theorem \ref{thm: Zd Classical} does not provide a characterization of deletion
tolerance and singularity for \eqref{eqn: Zd model} with i.i.d. Gaussian noise---namely, there is a gap in the characterization for $\alpha\in[1/2,d/2)$. In fact, it can be shown that
Propositions \ref{prop: Kakutani} and \ref{prop: variance linear statistics}/Remark \ref{rmk: covariance trick} cannot easily be used to
bridge this gap:

\begin{prop}
\label{prop: variance obstacle}
Let $\Pi$ be as in Theorem \ref{thm: Zd Classical}.
\begin{enumerate}
\item Recall the definitions of $P$ and $P_S$ in Proposition \ref{prop: Kakutani}. There exists an enumeration of $\mathbb Z^d=\{z_0,z_1,z_2,\ldots\}$
such that $P$ and $P_S$ are mutually absolutely continuous for every finite $S\subset\mathbb Z^d$ if and only if $\alpha<1/2$.
\item Let $f:\R\to\R$ be locally absolutely continuous, let $f_n$ be as in Remark \ref{rmk: scaling linear statistic}, and suppose there exists some nonempty interval $[a,b]\subset\R$ such that $\inf_{x\in[a,b]}|f'(x)|>0.$
Then, there exists some $C>0$ such that for every $n\geq1$,
\begin{align}
\label{Equation: linear statistics lower bound}
\limsup_{n\to\infty}\mathrm{Var}\left(\sum_{a\in\Pi}f_n(a)\right)\geq C\limsup_{n\to\infty}n^{d-2\alpha}.
\end{align}
In particular, this variance is unbounded whenever $\alpha<d/2$.
\end{enumerate}
\end{prop}
See Sections \ref{Proof of Variance Obstacle (1)} and \ref{Proof of Variance Obstacle (2)} for a proof. These considerations then lead us to the following general problem:

\begin{problem}
\label{Main Problem}
Which features of the point process $\Pi=\{V(z)+g_z\}_{z\in G}$ affect its rigidity?
\end{problem}

Given the limitations of Propositions \ref{prop: Kakutani} and \ref{prop: variance linear statistics}, new methods must be developed in order to address the above question. In addition to the obstacle posed by Proposition \ref{prop: variance obstacle}, both tools tend to be difficult to apply in cases where the noise has some complicated correlation structure. Kakutani's dichotomy requires independence to use and the calculations needed to bound the variance of test functions often become intractable when a complicated covariance structure is introduced. Previous works \cite{LamarreGhosalLiaoSpacialConditioning, LamarreGhosalLiaoSpectralRigidity} have found success using Proposition \ref{prop: variance linear statistics} in concert with the Feynman-Kac formula to give sufficient conditions for the number rigidity of various Schrödinger Eigenvalue Processes. Exponential test functions can be used to find tolerance thresholds for correlated eigenvalues in some cases, but these test function are not suitable for noise with heavy tailed distributions, and even for non-heavy tailed distributions the correlations must decay sufficiently quickly \cite{LamarreGhosalLiaoSpacialConditioning}.

\section{Main Result and Applications}

We now state our main result and some of its applications. We begin with some assumptions:

\begin{assumption}
\label{assumption: main}
Let $G$ be any countably infinite set. Suppose that $V:G\to\R$ satisfies the following conditions:
\begin{enumerate}[\quad a.]
\item The image $\mathrm{Im}(V)=\{x\in\R:V(z)=x\text{ for some }z\in G\}$
has no accumulation point.
\item For every $x\in \mathrm{Im}(V)$,
the set $\{z\in G:V(z)=x\}$
is finite.
\end{enumerate}
In particular, $V$ is unbounded.
Suppose that we denote
\begin{align}
\label{Equation: Image of V}
\mathrm{Im}(V)=\{t_0,t_1,t_2,\ldots\}\cup\{r_0,r_1,r_2,\ldots\},
\end{align}
where
$\cdots<t_2<t_1<t_0<0\leq r_0<r_1<r_2<\cdots$.
(Since $V$ is unbounded, at least one of the sets $\{t_0,t_1,t_2,\ldots\}$ or $\{r_0,r_1,r_2,\ldots\}$ must be infinite, but one of them will be finite if $V$ is bounded below or bounded above). We assume that $(g_z)_{z\in G}$ are random variables that satisfy the following:

\begin{enumerate}[\quad i.]
\item If $V$ is bounded below, then \begin{align}
\label{eqn: max limit 1}
\lim_{n\to\infty}\frac1{r_n-r_{n-1}}\max_{z\in G:V(z)\in[0,r_n]}|g_z|=0
\qquad\text{almost surely}.
\end{align}
\item If $V$ is bounded above, then
\begin{align}
\label{eqn: max limit 2}
\lim_{n\to\infty}\frac1{|t_n-t_{n-1}|}\max_{z\in G:V(z)\in[t_n,0)}|g_z|=0
\qquad\text{almost surely}.
\end{align}
\item If $V$ is neither bounded below nor bounded above, then both \eqref{eqn: max limit 1} and \eqref{eqn: max limit 2} hold.
\end{enumerate}
\end{assumption}

Our main result is as follows:

\begin{theorem}
\label{thm: Main}
Let $\Pi=\{V(z)+g_z\}_{z\in G}$.
If $G$, $V$, and $(g_z)_{z\in G}$
satisfy Assumption \ref{assumption: main},
then for every finite $S,T\subset G$ such that $|S|\neq|T|$, the processes
$\Pi_S$ and $\Pi_T$ are mutually singular.
\end{theorem}

\begin{remark}
\label{rmk: measurability initial}
When we say that two point processes are mutually singular, we mean that the {\it completion} of the probability distributions of the point processes on the set of counting measures (equipped with the Borel $\sigma$-algebra generated by vague convergence) are mutually singular. See Section \ref{Section: Proof of Main Resilt - Setup} for a detailed statement, and
Remark \ref{measurability concern} and Lemma \ref{addressing measurability} for more details on the necessity of the probability measures we consider being complete.
\end{remark}

\begin{remark}
If we take $T=\varnothing$ (meaning that $\Pi_T=\Pi$), then Theorem \ref{thm: Main} implies that $\Pi$ is deletion singular under Assumption \ref{assumption: main}.
\end{remark}

We now discuss applications and the optimality of Theorem \ref{thm: Main}.

\subsection{Applications}

Our main application consists of the following improvement of Theorem \ref{thm: Zd Classical}:

\begin{corollary}
\label{cor: new Zd}
Let $d\geq1$. Let $(g_z)_{z\in\Z^d}$ be any Gaussian process such that $g_z$ has mean zero and variance $\sigma^2>0$ for every $z$ (i.e., we make no assumption on the covariance). Consider the point process $\Pi=\{\norm{z}^\alpha+g_z\}_{z\in\Z^d}$, where $\norm{\cdot}$ is either the $\ell^1$ or the $\ell^\infty$ norm. If $\alpha>1$, then $\Pi$ is deletion singular.
\end{corollary}
\begin{proof}
In this case, the function $V(z)=\norm{z}^\alpha$ satisfies Assumption \ref{assumption: main} a. and b. Moreover, the image of $V$ can be written as $\{r_0,r_1,r_2,\ldots\}$ with $r_n=n^\alpha$. By the intermediate value theorem, it is easy to see that $r_n-r_{n-1}\geq cn^{\alpha-1}$ for some constant $c>0$. By a standard union bound/Borel-Cantelli argument, the fact that the $g_z$ are Gaussian with mean zero and variance $\sigma^2$ implies that
\[\limsup_{n\to\infty}\frac{1}{\sqrt{\log|\{z\in\Z^d,~\norm{z}\leq n\}|}}\max_{z\in\Z^d,~\norm{z}\leq n}|g_z|<\infty
\qquad\text{almost surely}.\]
Given that 
\begin{align}
\label{Equation: Log n^d is Log n}
\log|\{z\in\Z^d,~\norm{z}\leq n\}|=\log\big(O(n^d)\big)=O(\log n),
\end{align}
we get that $(g_z)_{z\in\Z^d}$ satisfies Assumption \ref{assumption: main} i. whenever $\alpha>1$.
\end{proof}

There are two features of this result that are most significant:
\begin{enumerate}
\item In Theorem \ref{thm: Zd Classical}, the lower bound on $\alpha$ for deletion singularity grows with $d$. In sharp contrast to this, our main result implies that we can replace this by the dimension-independent lower bound $\alpha>1$; this improves on $\alpha\geq d/2$ as soon as $d\geq3$. This works because the dimensional-dependence in our argument occurs as a multiplicative factor in a logarithmic term (i.e., \eqref{Equation: Log n^d is Log n}). Conversely, in the variances of linear statistics (i.e., \eqref{Equation: linear statistics lower bound}), the dimension occurs as a power law on $n$. This mechanism is crucial in the design of the proof of Theorem \ref{thm: Main}; see Section \ref{Section: Outline} for more details.
\item Suppose that we replace the i.i.d. Gaussians in Theorem \ref{thm: Zd Classical} with a centered stationary Gaussian process with a covariance decay of the form $|\E[g_zg_w]|\lesssim\norm{z-w}^{-\beta}$
as $\norm{z-w}\to\infty$ for some $\beta>0$. Then, Proposition \ref{prop: variance linear statistics} implies deletion singularity whenever
\[\alpha>\begin{cases}
d/2&\text{if }\beta>d,\\
d-\beta/2&\text{if }\beta<d.
\end{cases}\]
(See, e.g., \cite[Theorem 3.16]{LamarreGhosalLiaoSpacialConditioning} for a similar result). Thus, the sufficient conditions for deletion tolerance that one can get through linear statistics conditions will typically depend significantly on the correlation within the noise. This is because nontrivial covariances can significantly increase the size of the variance of linear statistics, because of the non-diagonal terms in the sum \eqref{eqn: Linear stat variance to covariance}. Thus,
in addition to dimension-independence, our result allows to provide correlation-independent sufficient conditions for deletion singularity. Again, our method of proof (outlined in Section \ref{Section: Outline}) was consciously designed to make this improvement: In the proof of Corollary \ref{cor: new Zd}, we bound the maxima of $|g_z|$ using a union bound, which is completely insensitive to correlations.
\end{enumerate}

\begin{remark}
\label{Remark: alpha=1}
Despite these improvements, it is interesting to note that Corollary \ref{cor: new Zd} does not improve on every case in Theorem \ref{thm: Zd Classical}. For instance, when $d=1$, Theorem \ref{thm: Zd Classical} implies that $\{|z|^\alpha+g_z\}_{z\in\Z}$ with $g_z$'s i.i.d. Gaussians is deletion singular whenever $\alpha\in(1/2,1]$. However, Corollary \ref{cor: new Zd} is only able to prove this when $\alpha>1$. That being said, it can be shown that the mechanism that we use to prove Corollary \ref{cor: new Zd} cannot be improved to $\alpha\leq1$ in general (i.e., with no assumption on the correlation between the $g_z$'s). See Example \ref{example: shifted lattice} for a counterexample in this direction.
\end{remark}

\begin{remark}
\label{rmk: ellp norm 2}
Let $\norm{\cdot}$ be the $\ell^p$ norm for some $p\in(1,\infty)$. Let
$0=s_0<s_1<s_2<\cdots$
denote all the possible values of $\norm{z}^p$ for $z\in\mathbb Z^d$; in other words, all numbers that can be written as sums of $p^{\mathrm{th}}$ powers of $d$ nonnegative integers.
If we let $(g_z)_{z\in\Z^d}$ be i.i.d. Gaussians, then using Proposition \ref{prop: Kakutani}, one can prove that $\Pi=\{\|z\|^\alpha+g_z\}_{z\in}$
is deletion tolerant whenever
$$\sum_{n=1}^\infty(s_n^{\alpha/p}-s_{n-1}^{\alpha/p})^2<\infty.$$
Conversely, as mentioned in Remark \ref{rmk: ellp norm 1}, Proposition \ref{prop: variance linear statistics} implies that $\Pi=\{\|z\|^\alpha+g_z\}_{z\in}$ is deletion singular whenever $\alpha\geq d/2$.

In this context,
it is natural to wonder whether one could improve this result in a way that is analogous to how
Corollary \ref{cor: new Zd} 
improves Theorem \ref{thm: Zd Classical}. If we replicate the argument we used in Corollary \ref{cor: new Zd}, then we get the following: If $(g_z)_{z\in\Z^d}$ is an arbitrary centered Gaussian process with $\E[g_z^2]=\sigma^2>0$ for all $z$, and $\norm{\cdot}$ is the $\ell^p$ norm, then $\Pi$ is deletion singular whenever
\begin{align}
\label{eqn: rigidity condition for general ell_p}
\limsup_{n\to\infty}\frac{\sqrt{\log n}}{s_n^{\alpha/p}-s_{n-1}^{\alpha/p}}=0.
\end{align}

The difficulty in interpreting these conditions is that, for general $p\in(1,\infty)$, calculating the possible outputs of $\norm{\cdot}^p$
is a very difficult problem.
This is true even for integer values of $p$, in which case calculating the $s_n$'s has connections to deep problems in number theory (see, e.g., the survey paper \cite{Waring}).
Nevertheless, it can be shown that, similarly to Corollary \ref{cor: new Zd}, the condition \eqref{eqn: rigidity condition for general ell_p} can provide dimension-independent criteria for the deletion singularity of $\Pi$, and thus improve on the condition $\alpha>d/2$. More specifically, it is known that for every $p\in(1,\infty)\cap\N$, there exists some critical dimension $d_\star\in\N$ large enough so that if $d\geq d_\star$, then $s_n=n$ for all $n\geq0$ (this is a classical result due to Hilbert \cite{Hilbert}). In that case, \eqref{eqn: rigidity condition for general ell_p} simplifies to
\begin{align}
\label{eqn: rigidity condition for general ell_p 2}
\limsup_{n\to\infty}\frac{\sqrt{\log n}}{n^{\alpha/p}-(n-1)^{\alpha/p}}=0;
\end{align}
\eqref{eqn: rigidity condition for general ell_p 2}  holds whenever $\alpha>p$.
\end{remark}

\subsection{Optimality}

In this section, we discuss the optimality of various conditions in Theorem \ref{thm: Main}. We start by giving an example to show that it is in fact necessary to assume $|S|\neq |T|$.

\begin{example}
\label{Ex: |S|=|T|}
Consider the point process $\Pi=\{\norm{z}^\alpha+g_z\}_{z\in\Z^d}$, where $\norm{\cdot}$ is the $\ell^1$ or $\ell^\infty$ norm and $(g_z)_{z\in \Z^d}$ are i.i.d gaussian with mean 0 and variance $\sigma^2$. Let $S,T \subset \Z^d$ be finite and such that $|S|=|T|$. Then for all $
\alpha>0$, $\Pi_S$ and $\Pi_T$ are mutually absolutely continuous. 
\end{example}

This implies that even in the best circumstances, we can only detect how many points were deleted from $\Pi$, not which specific points were deleted;
see Section \ref{Proof of S=T} for a proof.

Recall Assumption \ref{assumption: main} iii. We show by giving an example that the double sided condition in the case where $V$ is unbounded in both directions is necessary. Intuitively we can think of a process failing to fulfill either side of the condition as being able to ``hide" the deletion in the side of the process where the noise effects are too strong to be able to find deletions. 
\begin{example}
\label{Ex: double-sided counterexample}
Consider the point process on the real line \[\Pi=\{-z^2+g_z:z\in\Z^-\}\bigcup\{z^\alpha+g_z:z\in\Z^+\cup\{0\}\}\] where $(g_z)_{z\in\Z}$ are i.i.d gaussian with mean $0$ and variance $1$, and $\alpha <\frac{1}{2}$. Then, $\Pi$ and $\Pi_{\{0\}}$ are mutually absolutely continuous.
\end{example}

See Section \ref{Proof of doublesided} for a proof.

\noindent Finally, following-up on Remark \ref{Remark: alpha=1}, we show that Theorem \ref{thm: Main} cannot be used to prove the deletion singularity of processes at a level of generality similar to Corollary \ref{cor: new Zd} (i.e., with general covariances) below the threshold $\alpha=1$:

\begin{example}
\label{example: shifted lattice}
Consider $\Pi=\{z^\alpha+g\}_{z\in\N}$, where $g$ is any random variable such that $g$ and $g+k$ are mutually absolutely continuous for any fixed $k\in\Z$.
\begin{enumerate}
\item If $\alpha=1$, then $\Pi_{\{1,\ldots,m\}}$ and $\Pi_{\{1,\ldots,n\}}$ are mutually absolutely continuous for every  $m,n\geq1$. In particular, $\Pi$ is not deletion singular.
\item If $\alpha\in(0,1)\cup(1,\infty)$, then
$\Pi$ is deletion singular.
\end{enumerate}
\end{example}

See Section \ref{Proof of shifted lattice} for proof.

\section{Outline of Proof}
\label{Section: Outline}

We will only outline the proof of Theorem \ref{thm: Main} in the case where $(g_z)_{z\in G}$ satisfies Assumption \ref{assumption: main} i., as the other cases follow similarly.

Let $S,T\subset G$ be finite and such that $|S|\neq|T|$; we can
assume without loss of generality that $|S|<|T|$.
We want to find some function $f$ such that, almost surely, $f(\Pi_S)=0$
and $f(\Pi_T)\neq0$. To achieve this, consider the following class of functions: Given an integer $k\geq0$ and any countable set $A$ (e.g., $A=\Pi_S$ or $A=\Pi_T$), let

\begin{align}
\label{eqn: Main Functional}
f_k(A) = \inf\limits_{\substack{B\subset G\\|B|=k}}
~\inf\limits_{\substack{\psi:G\setminus{B}\to A\\\text{bijection}}}
~\limsup\limits_{n\to\infty}\frac{1}{r_n-r_{n-1}}
~\max_{\substack{z\in G\setminus B\\V(z)\in[0,r_n]}}|\psi(z)-V(z)|.
\end{align}
The rationale for
the design of the latter is as follows:

\begin{enumerate}
\item We make an apriori guess, $k\in\N$, that represents how many points we believe are missing from either $\Pi_S$ or $\Pi_T$; i.e., we are trying to guess the cardinalities of $S$ and $T$.

\item Once we have made this guess, we try to guess which $k$ points were deleted from $\Pi$ in $\Pi_S$ or $\Pi_T$. In the definition of $f_k$, this corresponds to choosing a subset $B\subset G$ with $|B|=k$, which matches the deleted points.

\item On the one hand, if $k=|S|$, then the infimum over all $B$ such that $|B|=k$ in $f_k$ will eventually find $B=S$. If $B=S$, then the infimum over $\psi:G\setminus S\to\Pi_S$ will include the ``identity" map
\[\psi(z)=V(z)+g_z,\qquad z\in G\setminus S.\]
In that case, we have the almost-sure limit
\[\limsup\limits_{n\to\infty}\frac{1}{r_n-r_{n-1}}
~\max_{\substack{z\in G\setminus B\\V(z)\in[0,r_n]}}|\psi(z)-V(z)|=\limsup\limits_{n\to\infty}\frac{1}{r_n-r_{n-1}}
~\max_{\substack{z\in G\setminus B\\V(z)\in[0,r_n]}}|g_z|=0\]
by Assumption \ref{assumption: main} i.
Thus, we conclude that $f_k(\Pi_S)=0$ almost surely.

\item On the other hand, if $|T|>k$, we show that no matter how we choose $B$ and $\psi$, we are forced to have an infinite chain of ``mismatches" between the points missing from $\Pi_T$ and our guess $B$, i.e., points $z$ such that $\psi(z)\neq V(z)+g_z$. Using this infinite chain of mismatches, we will be able to show that $f_k(\Pi_T)\geq1$ almost surely.
\end{enumerate}

In summary, if $|S|\neq|T|$, then evaluating $f_k(\Pi_S)$ and $f_k(\Pi_T)$ for $k=0,1,2,\ldots$ eventually yields some $f$ such that $f(\Pi_S)=0\neq f(\Pi_T)$ almost surely.

\begin{remark} 
\label{k always less than or equal}
In the above outline, we do not describe what happens when we evaluate $f_k$ in $\Pi_S$ for $|S|<k$.
However, by starting at $k=0$ and iterating upwards, we are able to avoid dealing with this situation altogether.
\end{remark}

\begin{remark}
\label{measurability concern}
By virtue of including an infimum over an uncountable set, we are unable to show that the functionals $f_k$ are measurable. Thus, the fact that
some of these functionals output different values for $\Pi_S$ and $\Pi_T$ does not necessarily imply
that the point processes are mutually singular. That being said, we are able
to get around this issue by completing
the distributions of the point processes;
see Lemma \ref{addressing measurability} for the details.
\end{remark}

This approach is partially inspired by the one used to prove Theorem 1.3 in \cite{PeresSly} (see Section 2.1 therein). In that paper, the authors characterize the tolerance and singularity of the perturbed lattice $\Pi=\{x+Y_x\}_{x\in\Z^d}$ where $Y_x\sim N_d(0,\sigma^2I)$ are independent $d$-dimensional normal random vectors with independent components of variance $\sigma^2$. Both our approach and the one in that paper rely on the following two-step strategy:
\begin{enumerate}
\item attempting to match points in a perturbed-lattice-type model to a set of ``unperturbed" deterministic lattice points, and then
\item examining the fluctuations between random points and their matched deterministic lattice points in order to determine if points were deleted.
\end{enumerate}

That being said, the mechanism that we exploit in step 2 of this strategy is entirely different from the one used in \cite{PeresSly}. On the one hand, our strategy relies on comparing the maximum of these mismatches with the growth rate of $V$'s possible outputs (as evidenced by the design of $f_k$ in \eqref{eqn: Main Functional}). \cite{PeresSly} on the other hand, uses an averaging technique to check for deletions. The authors show that if the variance of the perturbations is small enough they can distinguish between the average deviations of the points in the deleted lattice and those of the original lattice along a particular path. Despite similarity in philosophy, the differences in the approach combined with the different setting result in some additional flexibility. In particular, we do not require independence of the noise since we do not rely on greedy lattice animals (or similar mechanisms) and, crucially for our setting, we are also able to have multiple points in $\Pi$ that originate from the same $z\in G$ (i.e., the function $V(z)$ need not be injective). Furthermore, our approach allows us to not only distinguish between the original process $\Pi$ and some deleted process $\Pi_S$, but also between two different deleted processes $\Pi_S$ and $\Pi_T$ with $|S|\neq|T|$, resulting in a mechanism that can tell us precisely how many points were deleted.

\section{Proof of Main Result}

\subsection{General Setup}
\label{Section: Proof of Main Resilt - Setup}

Let $G=\{z_1,z_2,z_3,\ldots\}$ be an arbitrary enumeration of $G$.
Let $\R^\N=\{\boldsymbol\omega=(\omega_i)_{i\in\N}:\omega_i\in\R\}$ be the set of real-valued sequences, equipped with the product topology
generated by the Euclidean topology on $\R$. Let $\mathcal B(\R^\N)$ be the Borel $\sigma$-algebra generated by this topology. By virtue of being a countable product of Polish spaces, $\R^\Z$ is a Polish space.
On this space, we define the set of random variables $(g_{z_i})_{i\in\N}$
as $g_{z_i}(\boldsymbol\omega)=\omega_i$ for $i\in\N$; we then equip $(\R^\N,\mathcal B(\R^\N))$ with the probability measure $\P$  that corresponds to the joint distribution of the $g_z$'s.
Let $\mathcal L\subset\R^\N$ denote the set of sequences such that the set $|\{i\in\N:V(z_i)+\omega_i\in K\}|$  is finite for every compact $K\subset\mathbb R$. It is a simple exercise to show that $\mathcal L\in\mathcal B(\R^\N)$ (e.g., \cite[Appendix B]{LiREU});
as mentioned in the introduction, we assume that $\P[\mathcal L]=1$.

Let $\mathbb M$ be the space of locally finite Borel measures on $\R$, and let $\mathcal B(\mathbb M)$ be the Borel $\sigma$-algebra on that space generated by the vague topology (e.g., \cite[Page 109]{RandomMeasureTheoryandApplications}). Under this topology, $\mathbb M$ is also a Polish space (e.g., \cite[Theorem 4.2]{RandomMeasureTheoryandApplications}).
For every finite $S\subset G$, we define the map $\Pi_S:\R^\N\to\mathbb M$ as
\[\Pi_S(\boldsymbol \omega)=\left(\sum_{i\in\N:~z_i\not\in S}\delta_{V(z_i)+\omega_i}\right)\mathbf 1_{\{\boldsymbol\omega\in\mathcal L\}}
=\left(\sum_{i\in\N:~z_i\not\in S}\delta_{V(z_i)+g_{z_i}(\boldsymbol\omega)}\right)\mathbf 1_{\{\boldsymbol\omega\in\mathcal L\}},\]
where $\delta_x$ denotes the Dirac mass at $x\in\R$.
(This includes the case where $S=\varnothing$, which we simply denote $\Pi=\Pi_\varnothing$).
It is a straightforward exercise to show that $\Pi_S$ is measurable with respect to $\mathcal B(\R^\N)$ and $\mathcal B(\mathbb M)$ (e.g., \cite[Appendix B]{LiREU}).

Let $\mathbb P_{\Pi_S}=\P\circ\Pi_S^{-1}$ denote the probability distribution of $\Pi_S$ on $(\mathbb M,\mathcal B(\mathbb M))$, and let $\overline{\P}^0_{\Pi_S}$ be its completion, denoting the corresponding complete $\sigma$-algebra as $\mathcal{F}_{\Pi_S}$. Finally, to ensure that our distributions are operating over the same $\sigma$-algebra, we define \begin{align}
\label{equation: completion}
\mathcal{F}=\bigcap\limits_{S\subset G:|S|<\infty}\mathcal{F}_{\Pi_S},\end{align}
and we let $\overline{\P}_{\Pi_S}$ be the restriction of $\overline{\P}^0_{\Pi_S}$ on $\mathcal F$.

Following-up on Remark \ref{rmk: measurability initial},
our aim is to show that $\overline{\P}_{\Pi_S}$ and $\overline{\P}_{\Pi_T}$ are mutually singular in the space $(\mathbb M,\mathcal F)$ when $|S|\neq|T|$ are both finite.
Our proof of this consists of constructing disjoint measurable sets $A_S,A_T\in\mathcal F$ such that
$\overline\P_S[A_S]=1$ and $\overline\P_T[A_T]=1$. We now go through a case-by-case proof of this.

\subsection{Case 1. $V$ is Bounded Below.}

 For the sake of clarity we break the proof up into several lemmas that when combined give the desired result. We prove the theorem under (\ref{eqn: max limit 1}) and then describe in Section \ref{remaining cases} how the other two cases follow similarly.
For the rest of this section we can assume without loss of generality that $|S|<|T|$.
\begin{lemma}
\label{lem: f_k(Pi_S)=0 proof}
Let $f_k$ be defined as in \eqref{eqn: Main Functional}.
If $\boldsymbol\omega\in\mathcal L$, then
\[f_{|S|}(\Pi_S(\boldsymbol\omega))\leq\limsup\limits_{n\to\infty}\frac{1}{r_n-r_{n-1}}\max_{z\in G:V(z)\in[0,r_n]}|g_z(\boldsymbol\omega)|.\]
\end{lemma}

\begin{proof}
If $k=|S|$, then the function $f_k$ is able to make the correct guess for every point in $\Pi_S(\boldsymbol\omega)$ through the subset $B=S$ and the bijection $\psi':G\setminus S \to \Pi_S(\boldsymbol\omega)$ such that $\psi'(z) = V(z)+g_z(\boldsymbol\omega)$. Thus by \eqref{eqn: Main Functional},

\begin{align*}
f_k(\Pi_S(\boldsymbol\omega)) &\leq\inf\limits_{\psi:G\setminus{S}\to \Pi_S(\boldsymbol\omega)}\limsup\limits_{n\to\infty}\frac{1}{r_n-r_{n-1}}\max_{\substack{z\in G\setminus B\\V(z)\in[0,r_n]}}|\psi(z)-V(z)|\\ &\leq \limsup\limits_{n\to\infty}\frac{1}{r_n-r_{n-1}}\max_{\substack{z\in G\setminus B\\V(z)\in[0,r_n]}}|\psi'(z)-V(z)| \\&\leq\limsup\limits_{n\to\infty}\frac{1}{r_n-r_{n-1}}\max_{z\in G:V(z)\in[0,r_n]}|g_z(\boldsymbol\omega)|,
\end{align*}
as desired.
\end{proof}

\begin{lemma}
\label{lem: f_k(Pi_T)>1 proof}
Suppose that $V$ is bounded below, and that it satisfies Assumption \ref{assumption: main} a. and b. If $\boldsymbol\omega\in\mathcal L$, then
for every $0\leq k<|T|$, the function $f_k$ defined in \eqref{eqn: Main Functional} satisfies
\[f_k(\Pi_T(\boldsymbol\omega))\geq1-\limsup\limits_{n\to\infty}\frac{1}{r_n-r_{n-1}}\max_{z\in G:V(z)\in[0,r_n]}|g_z(\boldsymbol\omega)|.\]
\end{lemma}

\begin{proof}
Let $0\leq k<|T|$.
By definition of $f_k$ in \eqref{eqn: Main Functional}, it suffices to prove that for every choice of subset $B\subset G$ such that $|B|=k$ and bijection 
$\psi:G\setminus B\to\Pi_T(\boldsymbol\omega)$, one has
\begin{align}
\label{Equation: fk lower bound}
\limsup\limits_{n\to\infty}\frac{1}{r_n-r_{n-1}}\max_{\substack{z\in G\setminus B\\V(z)\in[0,r_n]}}|\psi(z)-V(z)|
\geq1-\limsup\limits_{n\to\infty}\frac{1}{r_n-r_{n-1}}\max_{z\in G:V(z)\in[0,r_n]}|g_z(\boldsymbol\omega)|.
\end{align}
Thus, for the remainder of this proof,
we assume that $B$ and $\psi$ are a fixed choice of such a subset and bijection.
The first ingredient in our proof of \eqref{Equation: fk lower bound} is to show that there exists an infinite sequence of distinct elements $u_0,u_1,u_2,\ldots\in G\setminus B$
that form a chain of mismatches, in the sense that $\psi(u_n)=V(u_{n+1})+g_{u_{n+1}}(\boldsymbol\omega)$ for all $n\geq0$.
One can generate this sequence using the following algorithm (see Figure \ref{Figure: Algorithm} in the appendix for an illustration):

Since $|T|>k=|B|$, there exists $u_0\in T\setminus B$. Since $V(u_0)+g_{u_0}(\boldsymbol\omega)\not\in\Pi_T(\boldsymbol\omega)$, this means that there exists some $u_1\in G\setminus T$ such that $u_1\neq u_0$ and $\psi(u_0)=V(u_1)+g_{u_1}(\boldsymbol\omega)$.
If $u_1\in G\setminus B$, then the fact that $\psi$ is a bijection means that there exists $u_2\in G\setminus T$ such that $u_2\not\in\{u_0,u_1\}$ and $\psi(u_1)=V(u_2)+g_{u_2}(\boldsymbol\omega)$.
With this in hand, we may now attempt to generate $u_n$ for $n\geq3$ using the following recursive procedure: For every $n\geq3$, if $u_{n-1}\in G\setminus B$, then use the fact that $\psi$ is a bijection to conclude that there exists $u_n\in G\setminus T$ such that $u_n\not\in\{u_0,\ldots,u_{n-1}\}$ and $\psi(u_{n-1})=V(u_n)+g_{u_n}(\boldsymbol\omega)$.

The only mechanism that could prevent this algorithm from generating an infinite sequence would be the existence of some $n_\star\geq1$ such that $u_{n_\star}\in B$. In such a case, the recursive step fails because $\psi(u_{n_\star})$ is not defined.
However, if that happens, then we can restart the algorithm by choosing a new initial point $u_0'\in T\setminus B$ that is different from $\{u_0,u_1,\ldots,u_{n_\star}\}$.
It is always possible to choose such a $u_0'$ because the algorithm ensures that $u_0\in T\setminus B$, $u_1,\ldots,u_{n_\star-1}\in G\setminus (T\cup B)$, and $u_{n_\star}\in B\setminus T$. 
In particular, the sequence $u_0,\ldots,u_{n_\star}$ contains one element from $T$ and one distinct element from $B$; hence there will be at least one element left in $T\setminus B$ because $|T|>|B|$.
Since $\psi$ is a bijection, the new sequence of mismatches $\psi(u_{n-1}')=V(u_n')+g_{{u_n'}}(\boldsymbol\omega)$ recursively generated from $u_0'$ will not contain any element in $\{u_0,\ldots,u_{n_\star}\}$.
In conclusion, since $|T|>|B|$, we can repeatedly restart the algorithm with a new initial point in $T\setminus B$ until we either generate an infinite sequence, or until we exhaust all the points in $B$ (in which case the next starting point is guaranteed to generate an infinite sequence since there are no more points in $B$ that can terminate it).

For the remainder of this proof, we thus assume that $u_0,u_1,u_2,\ldots\in G\setminus B$ are distinct and such that $\psi(u_n)=V(u_{n+1})+g_{u_{n+1}}(\boldsymbol\omega)$ for $n\geq0$. The second ingredient in our proof is to extract a subsequence $0\leq n_0<n_1<n_2<\cdots$ that satisfies the following properties:
\begin{align}
\label{Equation: Lower Bound Subsequence 1}
0\leq V(u_{n_0+1})<V(u_{n_1+1})<V(u_{n_2+1})<\cdots,
\end{align}
and
\begin{align}
\label{Equation: Lower Bound Subsequence 2}
0\leq V(u_{n_k})<V(u_{n_k+1})\qquad\text{for every }k\geq0.
\end{align}
Before we show that such a subsequence exists, we explain how it can be used to prove \eqref{Equation: fk lower bound}, and thus conclude the proof of Lemma \ref{lem: f_k(Pi_T)>1 proof}:
Let us denote $s_0,s_1,\cdots$ and $h_0,h_1,\ldots$ such that
\[r_{s_k}=V(u_{n_k+1})
\qquad\text{and}\qquad
r_{h_k}=V(u_{n_k}),\qquad k\geq0;\]
by \eqref{Equation: Lower Bound Subsequence 1} and \eqref{Equation: Lower Bound Subsequence 2}, 
$s_k<s_{k+1}$ and
$h_k<s_k$ for $k\geq0$.
Given this, we note that
\[\text{left-hand side of }\eqref{Equation: fk lower bound}
\geq
\limsup\limits_{k\to\infty}\frac{1}{r_{s_k}-r_{s_k-1}}\max_{\substack{z\in G\setminus B\\V(z)\in[0,r_{s_k}]}}|\psi(z)-V(z)|,\]
where we get the lower bound by taking the $\limsup$ along the subsequence $\{s_k\}_{k\geq1}$.
If we lower bound the maximum above by making the choice $z=u_{n_k}$ (which we can make because $V(u_{n_k})=r_{h_k}<r_{s_k}$), then this yields
\[\text{left-hand side of }\eqref{Equation: fk lower bound}
\geq
\limsup\limits_{k\to\infty}\frac{|\psi(u_{n_k})-V(u_{n_k})|}{r_{s_k}-r_{s_k-1}}.\]
By definition of $h_k$, $s_k$, and the sequence of mismatches $\{u_n\}_{n\geq1}$, we note that
\[\psi(u_{n_k})-V(u_{n_k})=V(u_{n_k+1})+g_{u_{n_k+1}}(\boldsymbol\omega)-V(u_{n_k})=r_{s_k}+g_{u_{n_k+1}}(\boldsymbol\omega)-r_{h_k}.\]
Therefore, we get
\[\text{left-hand side of }\eqref{Equation: fk lower bound}
\geq
\limsup\limits_{k\to\infty}\frac{|r_{s_k}+g_{u_{n_k+1}}(\boldsymbol\omega)-r_{h_k}|}{r_{s_k}-r_{s_k-1}}.\]
By the triangle inequality, this yields
\[\text{left-hand side of }\eqref{Equation: fk lower bound}
\geq
\limsup\limits_{k\to\infty}\frac{|r_{s_k}-r_{h_k}|-|g_{u_{n_k+1}}(\boldsymbol\omega)|}{r_{s_k}-r_{s_k-1}}.\]
Given that $h_k<s_k$, we have that
$r_{h_k}\leq r_{s_k-1}$; hence $|r_{s_k}-r_{h_k}|\geq r_{s_k}-r_{s_k-1}$. Therefore,
\begin{align}
\label{Equation: fk lower bound second to last}
\text{left-hand side of }\eqref{Equation: fk lower bound}
\geq
1-\liminf\limits_{k\to\infty}\frac{|g_{u_{n_k+1}}(\boldsymbol\omega)|}{r_{s_k}-r_{s_k-1}}.
\end{align}
Finally, if we use the fact that $z=u_{n_k+1}$ is such that $V(z)=r_{s_k}$, then we get that
\begin{multline*}\liminf\limits_{k\to\infty}\frac{|g_{u_{n_k+1}}(\boldsymbol\omega)|}{r_{s_k}-r_{s_k-1}}
\leq\limsup_{k\to\infty}\frac1{r_{s_k}-r_{s_k-1}}\max_{\substack{z\in G\setminus B\\V(z)\in[0,r_{s_k}]}}|g_z(\boldsymbol\omega)|\\
\leq\limsup_{n\to\infty}\frac{1}{r_n-r_{n-1}}\max_{\substack{z\in G\setminus B\\V(z)\in[0,r_n]}}|g_z(\boldsymbol\omega)|
\leq\limsup\limits_{n\to\infty}\frac{1}{r_n-r_{n-1}}\max_{z\in G:V(z)\in[0,r_n]}|g_z(\boldsymbol\omega)|.\end{multline*}
If we combine this with \eqref{Equation: fk lower bound second to last}, then we finally obtain \eqref{Equation: fk lower bound}.

We now wrap up the proof of Lemma \ref{lem: f_k(Pi_T)>1 proof} by constructing the subsequence that satisfies \eqref{Equation: Lower Bound Subsequence 1} and \eqref{Equation: Lower Bound Subsequence 2}. We claim that this can be done explicitly as follows: First, define
\begin{align}
\label{Equation: Subsequence Infimum 1}
n_0=\inf\{n\geq0:V(u_{n+1})>V(u_n)\geq0\},\end{align}
and then for $k\geq1$, if we are given $n_{k-1}$, then we define
\begin{align}
\label{Equation: Subsequence Infimum 2}
n_k=\inf\{n\geq n_{k-1}+1:V(u_{n+1})>V(u_{n_{k-1}+1})\}.
\end{align}
Indeed, if $n_0<n_1<n_2<\cdots$ are well-defined (in the sense that the corresponding infima are finite), then the only claim in \eqref{Equation: Lower Bound Subsequence 1} and \eqref{Equation: Lower Bound Subsequence 2} that is not immediate from these definitions
is that $V(u_{n_k})< V(u_{n_k+1})$ for all $k\geq1$. Toward this end, suppose by contradiction that
\begin{align}
\label{Equation: Ordered Sequence Contradiction}
V(u_{n_k})\geq V(u_{n_k+1})\qquad\text{for some }k\geq1.
\end{align}
By definition of $n_k$, we have that $V(u_{n_k+1})> V(u_{n_{k-1}+1})$; combining this
with \eqref{Equation: Ordered Sequence Contradiction} implies that $V(u_{n_k})>V(u_{n_{k-1}+1})$,
and thus $u_{n_k}\neq u_{n_{k-1}+1}$ (by definition of $n_k$, this is equivalent to $n_k> n_{k-1}+1$). However, if this were true, then $n=n_k-1\in[n_{k-1}+1,n_k)$ would be such that $V(u_{ n+1})=V(u_{n_k})>V(u_{n_{k-1}-1})$, which contradicts the definition of $n_k$ as the smallest integer in $[n_{k-1}+1,\infty)$ that satisfies this property.

It now only remains to show that the infima in \eqref{Equation: Subsequence Infimum 1} and \eqref{Equation: Subsequence Infimum 2} are finite.
Given that $V$ is bounded below, that $\mathrm{Im}(V)$ has no accumulation point (Assumption \ref{assumption: main} a.), and that $V$'s level sets are finite (Assumption \ref{assumption: main} b.), we conclude that for every $R\in\R$, the set $\{z\in G:V(z)< R\}$ is finite. Therefore, since the points $u_1,u_2,\ldots\in G$ are distinct, $V(u_n)\to\infty$ as $n\to\infty$. This immediately implies that \eqref{Equation: Subsequence Infimum 1} and \eqref{Equation: Subsequence Infimum 2} are finite.
\end{proof}

Now that we have shown that there exists some $k$ such that $f_k(\Pi_S(\boldsymbol\omega))=0\neq f_k(\Pi_T(\boldsymbol\omega))$ for all $\boldsymbol\omega$ in a probability-one event (i.e., $\mathcal L$ intersected with the event where \eqref{eqn: max limit 1} holds), it remains only to address the measurability concern of Remark \ref{measurability concern}. We begin with the following straightforward lemma about the completion of pushforward measures:

\begin{lemma}
\label{pushforward completion lemma}
Let $(\Omega,\mathcal{A},\mu)$ be a probability space, let $(\mathcal{M},\mathcal{B})$ be a measurable space, and let $T:\Omega\to\mathcal M$ be measurable. Suppose that $A\in\mathcal A$ is such that $\mu(A)=1$. If $T(A)$ is in the completion of $\mathcal B$ under $\mu\circ T^{-1}$, then $ \overline{\mu \circ T^{-1}}(T(A))=1$,
where $\overline\cdot$ denotes the completion of a measure.
\end{lemma}

\begin{proof}
Recall that by definition of completion, if $C$ is in the completion of $\mathcal B$, then \[\overline{\mu\circ T^{-1}}(C) = \inf\limits_{\substack{B\in\mathcal{B}\\C\subset B}}\{\mu\circ T^{-1}(B)\}.\] 
Thus, it suffices to show that for every $B\in\mathcal{B}$ such that $T(A)\subset B$ we have $\mu\circ T^{-1}(B)=1$. 

Toward this end, let $B\in\mathcal B$ be such that $T(A)\subset B$.
For every $\omega\in A$, we have $T(\omega)\in T(A)$ and so $T(\omega)\in B$. Then, $A\subset\{\omega\in\Omega:T(\omega)\in B\}$ and since $\mu(A)=1$, we have $\mu\circ T^{-1}(B)=\mu(\{\omega\in\Omega:T(\omega)\in B\})\geq\mu(A)=1$.
\end{proof}

We now conclude the proof of Theorem \ref{thm: Main} in the case where $V$ is bounded below with the following lemma:

\begin{lemma}
\label{addressing measurability}
Suppose that $V$ is bounded below, and that it satisfies Assumption \ref{assumption: main} a. and b. Suppose in addition that \eqref{eqn: max limit 1} holds.
The measures $\overline{\P}_{\Pi_S}$ and $\overline{\P}_{\Pi_T}$ are mutually singular.
\end{lemma}

\begin{proof}
Define the random variables $X$ on $(\R^\N,\mathcal B(\R^\N))$ as
\[X(\boldsymbol \omega)=\limsup\limits_{n\to\infty}\frac{1}{r_n-r_{n-1}}\max_{z\in G:V(z)\in[0,r_n]}|g_z(\boldsymbol\omega)|\]
($X$ is clearly Borel-measurable by virtue of each $g_z$ being Borel measurable).
Define the sets $A_S,A_T\subset\mathbb{M}$ as
\[A_S=\Pi_S(\{\boldsymbol \omega\in\R^\N:X(\boldsymbol \omega)=0\}\cap\mathcal L)
\qquad\text{and}\qquad
A_T=\Pi_T(\{\boldsymbol \omega\in\R^\N:X(\boldsymbol \omega)=0\}\cap\mathcal L).\]
By virtue of being the images of a Borel subset of the polish space $\R^\Z$ through Borel maps, $A_S$ and $A_T$ are analytic sets, hence universally measurable (i.e., measurable with respect to every completion of $\mathcal B(\mathbb M)$); see, e.g., \cite[Theorem 21.10]{SetTheory}. In particular,
$A_S,A_T\in\mathcal F$. By \eqref{eqn: max limit 1} and the assumption that $\mathcal L$ has probability one, $\P(\{\boldsymbol \omega\in\R^\N:X(\boldsymbol \omega)=0\}\cap\mathcal L)=1$. Therefore, it follows from Lemma \ref{pushforward completion lemma} that
\[\overline{\P}_{\Pi_S}(A_S)=\overline{\P\circ\Pi_S^{-1}}(\Pi_S(\{\boldsymbol \omega\in\R^\N:X(\boldsymbol \omega)=0\}\cap\mathcal L))=1,\]
and similarly $\overline{\P}_{\Pi_T}(A_T)=1$.

By Lemmas \ref{lem: f_k(Pi_S)=0 proof} and \ref{lem: f_k(Pi_T)>1 proof}, for every $\boldsymbol\omega\in\mathcal L$,
the maps $f_k$ satisfy the following:
\begin{align*}
\text{if $k=|S|$, then }& f_k(\Pi_S(\boldsymbol\omega))\leq X(\boldsymbol \omega);\\
\text{if $k<|T|$, then }& f_k(\Pi_T(\boldsymbol\omega))\geq 1-X(\boldsymbol \omega).
\end{align*}
Given that $A_S$ (resp. $A_T$) only includes images of $\boldsymbol\omega\in\mathcal L$ through $\Pi_S$ (resp. $\Pi_T$) for which $X(\boldsymbol\omega)=0$, this means that
\[A_S\subset\{\mu\in\mathbb M:f_{|S|}(\mu)\leq0\}
\qquad\text{and}\qquad
A_T\subset\{\mu\in\mathbb M:f_{|S|}(\mu)\geq1\}.\]
In particular, $A_S\cap A_T =\varnothing$, which concludes the proof.
\end{proof}

\subsection{Cases 2 and 3. $V$ is Bounded Above or Unbounded}
\label{remaining cases}
In the cases where $V$ is bounded above (and Assumption \ref{assumption: main} ii. holds) or where $V$ is neither bounded below nor bounded above (and Assumption \ref{assumption: main} iii. holds),
the proof of Theorem \ref{thm: Main} is nearly identical with minor changes that we illustrate below.

When $V$ is bounded above, we must first make a minor modification to the functional $f_k$ in order to account for the fact that the image of $V$ now has infinitely many negative outputs $\{t_n\}_{n\geq0}$ and only finitely many positive outputs. To do this, we define
\[\hat f_k(A) = \inf\limits_{\substack{B\subset G\\|B|=k}}~\inf\limits_{\substack{\psi:G\setminus{B}\to A\\\text{bijection}}}
~\limsup\limits_{n\to\infty}\frac{1}{|t_n-t_{n-1}|}
~\max_{\substack{z\in G\setminus B\\V(z)\in[t_n,0]}}|\psi(z)-V(z)|.\]

Using the functional $\hat f_k$, the proof follows identically to the one under Assumption \ref{assumption: main} i. In particular, we can repeat the proof of Lemma \ref{lem: f_k(Pi_S)=0 proof} using the exact same $\psi':G\setminus S\to\Pi_S(\boldsymbol \omega)$ such that $\psi'(z)=V(z)+g_z(\boldsymbol \omega)$ in order to get \[\hat f_{|S|}(\Pi_S(\boldsymbol \omega))\leq\limsup\limits_{n\to\infty}\frac{1}{|t_n-t_{n-1}|}\max_{z\in G:V(z)\in[t_n,0]}|g_z(\boldsymbol \omega)|,\qquad\boldsymbol\omega\in\mathcal L.\]

Similarly, by direct analogy to the proof of Lemma \ref{lem: f_k(Pi_T)>1 proof}, given any set $B\subset G$ with $B=k<|T|$ and a bijection $\psi$, we construct an infinite chain of mismatches $u_0,u_1,u_2,\ldots$ such that $\psi(u_n)=V(u_{n+1})+g_{u_{n+1}}(\boldsymbol \omega)$. Then by a symmetric argument to the one in the first case we can extract a subsequence satisfying the same properties as the one in the first case but going in the opposite direction (i.e. $V(u_{n_0+1})>V(u_{n_1+1})>V(u_{n_2+1})>\cdots $ and $0\geq V(u_{n_k})>V(u_{n_k+1})$). Using this monotone subsequence allows us to follow the same chain of inequalities as in Lemma \ref{lem: f_k(Pi_T)>1 proof} and get 
\[\hat f_k(\Pi_T(\boldsymbol \omega))\geq1-\limsup\limits_{n\to\infty}\frac{1}{|t_n-t_{n-1}|}\max_{z\in G:V(z)\in[t_n,0]}|g_z(\boldsymbol \omega)|,\qquad\boldsymbol\omega\in\mathcal L,\] 
for $0\leq k<|T|$.

At this point, if we restrict to the event where \eqref{eqn: max limit 2} holds, then we can find some $k\geq0$ such that $\hat f_k(\Pi_S)\leq0<1\leq\hat f_k(\Pi_T)$; the mutual singularity is then proved from this using Lemmas \ref{pushforward completion lemma} and \ref{addressing measurability}.

Under Assumption \ref{assumption: main} iii., we once again modify the functional, and consider $h_k = \max\{f_k,\hat f_k\}$. We immediately get by choosing $\psi':G\setminus S\to\Pi_S(\boldsymbol \omega)$ such that $\psi'(z)=V(z)+g_z(\boldsymbol \omega)$, that $h_{|S|}(\Pi_S(\boldsymbol \omega))\leq 0$ whenever $\boldsymbol\omega\in\mathcal L$ satisfies both \eqref{eqn: max limit 1} and \eqref{eqn: max limit 2} (this is clear since that inequality holds for both $f_{|S|}$ and $\hat f_{|S|}$ by the previous two cases). 

Once again using the same argument as in the proof of Lemma \ref{lem: f_k(Pi_T)>1 proof}, given a set $B\subset G$ with $B=k<|T|$ and a bijection $\psi$ we can construct an infinite chain of mismatches, $u_0,u_1,u_2,\ldots$ such that $\psi(u_n)=V(u_{n+1})+g_{u_{n+1}}(\boldsymbol \omega)$. The key difference in this case comes when we construct the subsequence that we use in the chain of inequalities. Because $V$ is unbounded in both directions we do not know a priori whether we can construct an infinite subsequence that is increasing or decreasing (i.e. $V(u_{n_{i+1}+1})>V(u_{n_i+1})$,
as in Case 1., or $V(u_{n_{i+1}+1})<V(u_{n_i+1})$, as in Case 2.). Nevertheless, Assumption \ref{assumption: main} a. and b. still implies that at least one of the following claims must hold:
\[\limsup_{n\to\infty}V(u_n)=\infty\qquad\text{or}\qquad
\liminf_{n\to\infty}V(u_n)=-\infty.\]
In particular, it is clear by direct analogy to the arguments in the proof of Lemma \ref{lem: f_k(Pi_T)>1 proof} that at least one of these subsequences (i.e., increasing or decreasing) must exist. This means that either $f_k(\Pi_T(\boldsymbol \omega)) \geq 1$ or $\hat{f}_k(\Pi_T(\boldsymbol \omega))\geq1$ whenever $\boldsymbol\omega\in\mathcal L$ satisfies both \eqref{eqn: max limit 1} and \eqref{eqn: max limit 2}, which immediately gives us the desired $h_k(\Pi_T(\boldsymbol \omega))\geq1$.
We then get the mutual singularity by
Lemmas \ref{pushforward completion lemma} and \ref{addressing measurability}. 

\section{Proofs of Remaining Results}

\subsection{Proof of Theorem \ref{thm: Zd Classical}}
\label{Proof of Zd classical}
The deletion tolerance of $\Pi$ for $\alpha<1/2$ follows immediately as a corollary to the first part of Proposition \ref{prop: variance obstacle}; we will prove this in the next section.

The deletion singularity of $\Pi$ for $\alpha>d/2$ is proved using the following application of Proposition \ref{prop: variance linear statistics}: Suppose that the $g_z$'s are i.i.d. copies of the random variable $g$.
Using the collection of test functions $f_n(x)=\mathrm{e}^{-x/n}$ we get: 
\begin{align*}
    \mathrm{Var}[f_{n}(\Pi)]=\sum\limits_{m\in\Z^d}\mathrm{Var}[f_n(\norm{m}^\alpha+g)]&\leq C\sum\limits_{m\in\N}m^{d-1}\mathrm{Var}[f_n(m^\alpha+g)]  \\& =C\mathrm{e}^{\sigma^2/n^2}(\mathrm{e}^{\sigma^2/n^2}-1)\sum\limits_{m\in \N}m^{d-1}\mathrm{e}^{-2m^\alpha/n},
\end{align*}
where we used the variance of the lognormal distribution and $|\{k\in\Z^d:\norm{k}=m\}|\leq Cm^{d-1}$ for a positive constant $C$ depending only on $d$.

We then have:
\begin{align*}
    C\mathrm{e}^{\sigma^2/n^2}(\mathrm{e}^{\sigma^2/n^2}-1)\sum\limits_{m\in \N}m^{d-1}\mathrm{e}^{-2m^\alpha/n} & =C\mathrm{e}^{\sigma^2/n^2}(\mathrm{e}^{\sigma^2/n^2}-1)\sum\limits_{m\in \frac{\N}{n^{1/\alpha}}}(mn^{1/\alpha})^{d-1}\mathrm{e}^{-2m^\alpha} \\& = C\mathrm{e}^{\sigma^2/n^2}(\mathrm{e}^{\sigma^2/n^2}-1)n^{(d-1)/\alpha}n^{1/\alpha}\frac{1}{n^{1/\alpha}}\sum\limits_{m\in \frac{\N}{n^{1/\alpha}}}m^{d-1}\mathrm{e}^{-2m^\alpha} \\&= C\mathrm{e}^{\sigma^2/n^2}(\mathrm{e}^{\sigma^2/n^2}-1)n^{d/\alpha}\frac{1}{n^{1/\alpha}}\sum\limits_{m\in \frac{\N}{n^{1/\alpha}}}m^{d-1}\mathrm{e}^{-2m^\alpha}
\end{align*}
Then, $\lim\limits_{n\to\infty}C\mathrm{e}^{\sigma^2/n^2}=C$ and by l'H\^{o}pital's rule, \[\lim\limits_{n\to\infty}(\mathrm{e}^{\sigma^2/n^2}-1)n^{d/\alpha}=\lim\limits_{n\to\infty}\frac{(\mathrm{e}^{\sigma^2/n^2}-1)}{n^{-d/\alpha}} = \lim\limits_{n\to\infty}\frac{2\sigma^2\alpha\mathrm{e}^{\sigma^2/n^2}}{dn^{2-d/\alpha}}\] which we can see goes to $0$ if $\alpha>d/2$.

Finally, using a Riemann sum, \[\lim\limits_{n\to\infty}\frac{1}{n^{1/\alpha}}\sum\limits_{m\in \frac{\N}{n^{1/\alpha}}}m^{d-1}\mathrm{e}^{-2m^\alpha}=\int\limits_0^\infty x^{d-1}\mathrm{e}^{-2x^\alpha}dx\]
which is finite for all $\alpha>0$.

Thus, $\lim\limits_{n\to\infty}\mathrm{Var}[f_{n}(\Pi)]=0$ for $\alpha>d/2$ and we have deletion singularity for $\Pi$ if $\alpha>d/2$ by Proposition \ref{prop: variance linear statistics}. \qed

\subsection{Proof of Proposition \ref{prop: variance obstacle} (1)}
\label{Proof of Variance Obstacle (1)}

Let $S\subset\Z^d$ with $|S|=n$. We start by enumerating $\Z^d=\{z_0,z_1,z_2,\ldots\}$ in non-decreasing order of norm and take the random vector $P=(\norm{z_i}^\alpha+g_{i})_{i\in\N} \in \R^{\N}$ with components given by $\Pi$ ordered by our chosen enumeration. We would like to compare this with the random vector $P_S$ that has components given by $\Pi_S$ ordered by the same enumeration of $\Z^d$. However, if we simply remove the entries in $P$ that came from $S\subset G$, we would be left with null coordinates in the vector $P_S$. In order to avoid null entries where the deletions would be, we must first shift the coordinates of $P_S$ to align them with $P$. 

To that end, we define the non-decreasing, surjective, deletion counting function $s:\N\to\{0,...,n\}$
\[s(i) = | S\cap\{z_0,...,z_i\}|.\]
Using this function, we define the vector $P_S=(\norm{z_{i+s(i)}}^\alpha+g_{i+s(i)})_{i\in\N}\in\R^{\N}$ that has components given by $\Pi_S$ ordered by our chosen enumeration of $\Z^d$.

We make one final simplification, in order to center the coordinates of our random vectors, by applying the pushforward map  \[\Xi((a_i)_{i\in\N})= (a_{i}+\norm{z_i}^\alpha)_{i\in\N}\] to both vectors. Then, $\Xi(X)$ gives the vector $P$ and $\Xi(Y)$ gives the vector $P_S$ where 
\[X=(g_{i})_{i\in\N} \quad Y=(g_{i+s(i)}+\norm{z_{i+s(i)}}^\alpha-\norm{z_i}^\alpha)_{i\in\N}\]

Since pushforwards preserve absolute continuity it suffices to show that the vectors $X$ and $Y$ are mutually absolutely continuous if and only if $\alpha<1/2$. Recall that $\{g_i\}_{i\in\N}$ were taken to be i.i.d so we can rewrite $Y$ equivalently as $(g_{i}+\norm{z_{i+s(i)}}^\alpha-\norm{z_i}^\alpha)_{i\in\N}$ in order to simplify notation.

Notice that $Y$ is simply a shift of $X$ since $Y_i=X_i+(\norm{z_{i+s(i)}}^\alpha-\norm{z_i}^\alpha)$. Then it follows from \cite[Theorem 1]{Shepp} that $X$ and $Y$ are mutually absolutely continuous if and only if the difference of their means converges in $\ell^2$.

Therefore, it suffices for mutual absolute continuity to show that the sum 
\begin{align} \label{eqn:l^2 sum}
\sum\limits_{i=0}^\infty(\norm{z_{i+s(i)}}^\alpha-\norm{z_i}^\alpha)^2
\end{align}
converges if and only if $\alpha<1/2$. 

Now notice that when $s(i)=1$, for each $k\in\N$ and all $i$ such that $\norm{z_i}=k$, we have that $\norm{z_i}^\alpha$ and $\norm{z_{i+1}}^\alpha$ can differ for at most $1$ value of $i$ and that the difference will be $(k+1)^\alpha-k^\alpha$. Similarly, when $s(i)=2$, for each $k\in\N$ and all $i$ such that $\norm{z_i}=k$ we have that $\norm{z_i}^\alpha$ and $\norm{z_{i+2}}^\alpha$ can differ for at most $2$ values of $i$ and that the difference will be at most $(k+2)^\alpha-k^\alpha$. Continuing up to $s(i)=n$, for each $k\in\N$ and all $i$ such that $\norm{z_i}=k$ we have that $\norm{z_i}^\alpha$ and $\norm{z_{i+n}}^\alpha$ can differ for at most $n$ values of $i$ and that the difference will be at most $(k+n)^\alpha-k^\alpha$. Finally, by definition of $s(i)$, there exists $M$ sufficiently large such that $s(i)=n$ for all $i>M$. We can now bound the sum as follows:
\[\sum\limits_{i=0}^\infty(\norm{z_{i+s(i)}}^\alpha-\norm{z_i}^\alpha)^2\leq n^2\sum\limits_{i=0}^\infty((i+n)^\alpha-i^\alpha)^2\] which converges for $\alpha<1/2$ by an immediate application of the mean value theorem.

We've now shown that for $\alpha<1/2$, $P$ and $P_S$ are mutually absolutely continuous for any $S\subset \Z^d$ satisfying $|S|<\infty$. Notice that for the specific choice, $S=\{0\}$, the sum \eqref{eqn:l^2 sum} reduces to $\sum\limits_{i=0}^\infty((i+1)^\alpha-i^\alpha)^2$ which converges only if $\alpha<1/2$. Thus we get that $P$ and $P_S$ are mutually absolutely continuous for every finite cardinality $S$ if and only if $\alpha<1/2$. \qed

\subsection{Proof of Proposition \ref{prop: variance obstacle} (2)}
\label{Proof of Variance Obstacle (2)}

Throughout this proof, we let $C$ denote a positive constant independent of $n$ whose exact value may change from line to line.
Let the $g_z$'s be i.i.d. copies of $g$.
By passing to the subsequence $r_n=n^\alpha$ we can bound the variance of the linear statistic as follows:

\begin{align*}
\mathrm{Var}[f_{r_n}(\Pi)]=\sum\limits_{m\in\Z^d}\mathrm{Var}[f((\frac{\norm{m}}{n})^\alpha+\frac{g}{n^\alpha})]&\geq C\sum\limits_{m\in\N}m^{d-1}\mathrm{Var}[f((\frac{m}{n})^\alpha+\frac{g}{n^\alpha})]\\ &\geq n^{-2\alpha}C\sum\limits_{m\in\N}m^{d-1}\E[f'((\frac{m}{n})^\alpha+\frac{g}{n^\alpha})^2],
\end{align*}
where the first inequality follows from the fact that there is on the order of $n^{d-1}$ points $z\in\Z^d$ such that $\norm{z}=n$ as $n\to\infty$, and the final inequality follows from \cite[Proposition 3.2]{VarianceBound}. We know that there exists an interval $[a,b]\subset\R$ and a real constant $\delta>0$ such that $f'(x)^2\geq\delta$ for all $x\in[a,b]$. Therefore we get:
\begin{align*}
\mathrm{Var}[f_{r_n}(\Pi)]  \geq &Cn^{-2\alpha}\sum\limits_{m\in\N}m^{d-1}\E[f'((\frac{m}{n})^\alpha+\frac{g}{n^\alpha})^2] \\ \geq &Cn^{-2\alpha}\sum\limits_{m\in\N} m^{d-1}\P[(\frac{m}{n})^\alpha+\frac{g}{n^\alpha}\in[a,b]]  \\\geq &Cn^{-2\alpha}\sum_{m=\lceil a^{1/\alpha}n\rceil}^{\lfloor b^{1/\alpha}n\rfloor} n^{d-1}\P[g\in[n^{\alpha}a-m^\alpha,n^{\alpha}b-m^\alpha]]  \\\geq &C\delta n^{-2\alpha}n^d,
\end{align*}
where the first inequality is the same as in the previous display; the second inequality follows from
$\mathbb E[f'(X)^2]\geq\mathbb E[\delta\mathbf 1_{[a,b]}(X)]=\delta\mathbb P[X\in[a,b]]$;
the third inequality follows from restricting which values of $m$ we are summing over, then taking $m\geq\lceil a^{1/\alpha}n\rceil$, then noting that
\[(\frac{m}{n})^\alpha+\frac{g}{n^\alpha}\in[a,b]\iff g\in[n^\alpha a-m^\alpha,n^\alpha b-m^\alpha];\]
and finally the last inequality follows from the fact that
\begin{multline*}\inf_{\substack{y\in [an^\alpha,bn^\alpha]\\n\geq1}}\mathbb P[g\in[n^{\alpha}a-y,n^{\alpha}b-y]]=\inf_{\substack{y\in [an^\alpha,bn^\alpha]\\n\geq1}}\int_{n^{\alpha}a-y}^{n^{\alpha}b-y}\frac{e^{-x^2/2\sigma^2}}{\sqrt{2\pi\sigma^2}}dx
=\int_{0}^{(b-a)}\frac{e^{-x^2/2\sigma^2}}{\sqrt{2\pi\sigma^2}}dx>0\end{multline*}
because the standard Gaussian density function is symmetric decreasing, and that the number of points $m\in\N$ such that 
$\lceil a^{1/\alpha}n\rceil\leq m\leq \lfloor b^{1/\alpha}n\rfloor$ is bounded below by a constant times $n$.
The proof is thus complete. \qed

\subsection{Proof of Example \ref{Ex: |S|=|T|}}
\label{Proof of S=T}
This proof follows very similarly to Section \ref{Proof of Variance Obstacle (1)} again employing Proposition \ref{prop: Kakutani}.

We consider the same enumeration of $\Z^d$ as in Section \ref{Proof of Variance Obstacle (1)} where we enumerate elements in non-decreasing order of norm.

Like before we begin with some simplifications. Following the notation of Proposition \ref{prop: Kakutani}, we consider the random vectors $P_S$ and $P_T$ in $\R^{\N}$ with components generated by the processes and apply a shift that allows us to align the vectors and ensure that there are no null coordinates where the deletions would be:

Define two non-decreasing, surjective, deletion counting function $s,t:\N\to\{0,...,n\}$
\[s(i) = | S\cap\{z_0,...,z_i\}| \quad t(i) = | T\cap\{z_0,...,z_i\}|.\]

Like in Section \ref{Proof of Variance Obstacle (1)}, we use $s$ and $t$ to shift the coordinates of the vectors $P_S$ and $P_T$ to get \[P_S=(\norm{z_{i+s(i)}}+g_{i+s(i)})_{i\in\N}\quad P_T= (\norm{z_{i+t(i)}}+g_{i+t(i)})_{i\in\N}\]

Then we apply the same centering pushforward map we used before:
\[\Xi((a_i)_{i\in\N})= (a_{i}+\norm{z_i}^\alpha)_{i\in\N}\]

and get that $\Xi(X) = P_S$ and $\Xi(Y)= P_T$ where
\[X= (g_{i+s(i)}+\norm{z_{i+s(i)}}^\alpha-\norm{z_i}^\alpha)_{i\in\N}\quad Y=(g_{i+t(i)}+\norm{z_{i+t(i)}}^\alpha-\norm{z_i}^\alpha)_{i\in\N}\]

Once again $Y$ is simply a shift of $X$ and so we can apply \cite[Theorem 1]{Shepp} to get that $X$ and $Y$ are mutually absolutely continuous if and only if the $\ell^2$ difference in their means is finite. Hence, we consider 
\begin{align}
\sum\limits_{i=0}^\infty(\norm{z_{i+s(i)}}^\alpha-\norm{z_{i+t(i)}}^\alpha)^2
\end{align}

However, this sum is clearly finite since by definition of $s$ and $t$, there exists $M\in\N$ such that for all $i>M$ we have $s(i)=t(i)=n$ and so 
\[
\sum\limits_{i=0}^\infty(\norm{z_{i+s(i)}}^\alpha-\norm{z_{i+t(i)}}^\alpha)^2=\sum\limits_{i=0}^M(\norm{z_{i+s(i)}}^\alpha-\norm{z_{i+t(i)}}^\alpha)^2<\infty
\]
Thus, $P_S$ and $P_T$ are mutually absolutely continuous whenever $|S|=|T|<\infty$. \qed

\subsection{Proof of Example \ref{Ex: double-sided counterexample}}
\label{Proof of doublesided}
We wish to show that despite fulfilling the growth condition for Theorem \ref{thm: Main} in the negative direction, this point process is nonetheless not deletion singular because the process fails to meet the condition in the positive direction (i.e, $V$ is unbounded both above and below and $\Pi$ satisfies (\ref{eqn: max limit 2}) but not (\ref{eqn: max limit 1})).

 We prove this by showing that $\Pi$ and $\Pi_{\{0\}}$ are mutually absolutely continuous by making a similar argument using Proposition \ref{prop: Kakutani} that we've done twice now in Sections \ref{Proof of Variance Obstacle (1)} and \ref{Proof of S=T}

 Following the notation of Proposition \ref{prop: Kakutani}, we let $P$ and $P_{\{0\}}$ be the random vectors in $\R^{\N}$ with components generated by the processes $\Pi$ and $\Pi_{\{0\}}$ respectively. We wish to shift the positive elements of $\Pi_{\{0\}}$ to cover the null entry caused by deleting $0$, while leaving the negative elements in their place. As such, we let 
 \begin{align*}
 P &= (0+g_0,1^\alpha+g_1,-1^2+g_{-1},2^\alpha+g_2,-2^2+g_{-2},\ldots), \\\\
P_{\{0\}} &=(1^\alpha+g_1,2^\alpha+g_2,-1^2+g_{-1},3^\alpha+g_3,-2^2+g_{-2},\ldots).
 \end{align*}

Applying \cite[Theorem 1]{Shepp}, we get that $P$ and $P_{\{0\}}$ are mutually absolutely continuous if the $\ell^2$ difference in their means converges, i.e.,
\[\sum\limits_{n=0}^\infty((n+1)^\alpha-n^\alpha)^2<\infty.\] This clearly holds for $\alpha<1/2$ and thus $P$ and $P_{\{0\}}$ are mutually absolutely continuous for $\alpha<1/2$. Finally, the absolute continuity of $\Pi$ and $\Pi_{\{0\}}$ follows by Proposition \ref{prop: Kakutani}. \qed

\subsection{Proof of Example \ref{example: shifted lattice}}
\label{Proof of shifted lattice}
For $\alpha=1$ we have the processes $\Pi_{\{1,...,m\}} = \sum\limits_{k>m}\delta_{k+g}$ and $\Pi_{\{1,...,n\}}=\sum\limits_{k>n}\delta_{k+g}$. Note that, by a simple change of variables, we can rewrite these as
\[\Pi_{\{1,...,m\}} =\sum_{k=1}^\infty\delta_{k+(g+m)}
\qquad\text{and}\qquad
\Pi_{\{1,...,n\}}=\sum_{k=1}^\infty\delta_{k+(g+n)}.\]
Our assumption that $g$ and $g+k$ are mutually absolutely continuous for $k\in\Z$ implies that $g+m$ and $g+n$ are mutually absolutely continuous for any fixed $m,n\in\Z$; it immediately follows that the processes $\Pi_{\{1,...,m\}}$ and $\Pi_{\{1,...,n\}}$ are mutually absolutely continuous.

Now for $\alpha\in(0,1)\cup(1,\infty)$, we will construct a measurable function $f$ that distinguishes between the processes $\Pi$ and $\Pi_S$ for any $S\subset\N$. For a countable set $A=\{x_i\}_{i\in\N}\subset\R$ enumerated in non-decreasing order of $x_i$, define $f_n(A)=|x_{n+1}-x_n|$. Then for every $n\in\N$, we have that $f_n(\Pi)=(n+1)^\alpha-n^\alpha$. When computing $f_n(\Pi_S)$ we consider two cases. When $S=\{1,\ldots,k\}$, we see that $f_{1}(\Pi_S) = (k+2)^\alpha-(k+1)^\alpha>f_1(\Pi)=2^\alpha-1$. For any other $S$, there exists $n\in S$ such that $n-1\notin S$. Then $f_{n-1}(\Pi_S)>n^\alpha-(n-1)^\alpha=f_{n-1}(\Pi)$. It follows that $\Pi$ and $\Pi_S$ are singular for every choice of $S$ and $\Pi$ is deletion singular. \qed

\newpage

\appendix
\section{Illustration of the Chain of Mismatches}

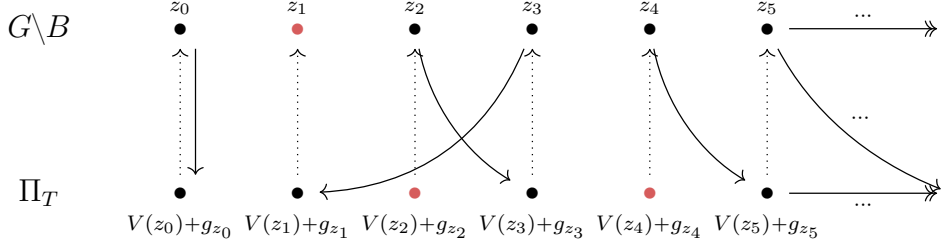
\begin{figure}[hbp]
\begin{center}
\begin{tikzcd}
	{G\setminus B} & {\bullet } & \textcolor{rgb,255:red,214;green,92;blue,92}{\bullet} & \bullet & \bullet & \bullet & \bullet && {\text{}} \\
	\\
	{\Pi_T} & {\bullet } & \bullet & \textcolor{rgb,255:red,214;green,92;blue,92}{\bullet} & \bullet & \textcolor{rgb,255:red,214;green,92;blue,92}{\bullet} & \bullet && {\text{}}
	\arrow["{z_0}"', shift right=4, draw=none, from=1-2, to=1-2, loop, in=55, out=125, distance=10mm]
	\arrow[shift left=2, from=1-2, to=3-2]
	\arrow["{z_1}"', shift right=4, draw=none, from=1-3, to=1-3, loop, in=55, out=125, distance=10mm]
	\arrow["{z_2}"', shift right=4, draw=none, from=1-4, to=1-4, loop, in=55, out=125, distance=10mm]
	\arrow[curve={height=12pt}, from=1-4, to=3-5]
	\arrow["{z_3}"', shift right=4, draw=none, from=1-5, to=1-5, loop, in=55, out=125, distance=10mm]
	\arrow[curve={height=-24pt}, from=1-5, to=3-3]
	\arrow["{z_4}"', shift right=4, draw=none, from=1-6, to=1-6, loop, in=55, out=125, distance=10mm]
	\arrow[curve={height=12pt}, from=1-6, to=3-7]
	\arrow["{z_5}"', shift right=4, draw=none, from=1-7, to=1-7, loop, in=55, out=125, distance=10mm]
	\arrow["\cdots", two heads, from=1-7, to=1-9]
	\arrow["\cdots", curve={height=12pt}, from=1-7, to=3-9]
	\arrow[dotted, from=3-2, to=1-2]
	\arrow["\footnotesize{{V(z_0)+g_{z_0}}}"', shift left=4, draw=none, from=3-2, to=3-2, loop, in=300, out=240, distance=5mm]
	\arrow[dotted, from=3-3, to=1-3]
	\arrow["\footnotesize{{V(z_1)+g_{z_1}}}"', shift left=4, draw=none, from=3-3, to=3-3, loop, in=300, out=240, distance=5mm]
	\arrow[dotted, from=3-4, to=1-4]
	\arrow["\footnotesize{{V(z_2)+g_{z_2}}}"', shift left=4, draw=none, from=3-4, to=3-4, loop, in=300, out=240, distance=5mm]
	\arrow[dotted, from=3-5, to=1-5]
	\arrow["\footnotesize{{V(z_3)+g_{z_3}}}"', shift left=4, draw=none, from=3-5, to=3-5, loop, in=300, out=240, distance=5mm]
	\arrow[dotted, from=3-6, to=1-6]
	\arrow["\footnotesize{{V(z_4)+g_{z_4}}}"', shift left=4, draw=none, from=3-6, to=3-6, loop, in=300, out=240, distance=5mm]
	\arrow[dotted, from=3-7, to=1-7]
	\arrow["\footnotesize{{V(z_5)+g_{z_5}}}"', shift left=4, draw=none, from=3-7, to=3-7, loop, in=300, out=240, distance=5mm]
	\arrow["\cdots"', two heads, from=3-7, to=3-9]
\end{tikzcd}
\caption{For sake of clarity, we enumerate the elements of $G$ as $z_0,z_1,z_2,\ldots$, and we drop the dependence of $\Pi_T$ and $g_z$ on $\boldsymbol\omega$ in the figure above.
The black dots on the top row represent the elements of $G\setminus S$, and the black dots on the bottom row represent the elements of $\Pi_T(\boldsymbol \omega)$; in both cases, red dots are deleted.
(In the above example, we delete $B=\{z_1\}$ from $G$, and we delete $T=\{z_2,z_4\}$ from $\Pi$.) The arrows pointing downward represent the outputs of the bijection $\psi$ (e.g., $\psi(z_0)=V(z_0)+g_{z_0}(\boldsymbol \omega)$, and $\psi(z_2)=V(z_3)+g_{z_3}(\boldsymbol \omega)$).
If we begin the algorithm described in the proof of Lemma \ref{lem: f_k(Pi_T)>1 proof} with the initial value $u_0=z_2\in T\setminus B$, then we get the finite sequence $(u_0,u_1,u_2)=(z_2,z_3,z_1)$. (This sequence can be visualized by following along the path traced by alternating dotted arrows and full arrows, starting from $V(u_0)+g_{u_0}(\boldsymbol \omega)$.) This sequence terminates at $n_\star=2$ because $u_2=z_1\in B$ is not mapped anywhere by $\psi$. Once that happens, we can now restart the algorithm with the new initial point $u_0'=z_4\in T\setminus B$. This point cannot have appeared in the previous sequence $u_0,u_1,u_2$, because it is not in the image of $\psi$. Moreover, Because all points in $B=\{z_1\}$ have been exhausted by the previous iteration of the algorithm, the algorithm is now guaranteed to generate an infinite sequence.} 
\label{Figure: Algorithm}
\end{center}
\end{figure}

\newpage
\bibliographystyle{alpha}
\bibliography{Bibliography}

 \end{document}